\documentclass[10pt]{amsart}
\usepackage{color}
\usepackage{tipa}
\usepackage{amssymb}
\usepackage{stmaryrd}
\usepackage{amsmath,amsfonts,amsthm,mathrsfs,bbm,array,subfigure}

\usepackage{booktabs}
\usepackage{url}

\newtheorem{theorem}{Theorem}[section]

\newtheorem{proposition}[theorem]{Proposition}

\newtheorem{problem}[theorem]{Problem}

\theoremstyle{definition}

\theoremstyle{remark}
\newtheorem{remark}[theorem]{Remark}

\numberwithin{equation}{section}

\newcommand{\Q}{{\mathbb Q}}
\newcommand{\Pn}{{\mathbb P^n}}
\newcommand{\Sl}{{\mathcal{S}_{\ell}}}
\newcommand{\Sln}{{\mathcal{S}}_{\ell}^{(n)}}

\begin{document}
	
    \title[Partitions and Elliptic Curves]{Partitions into triples with equal products and Families of Elliptic Curves}

	\author{Ahmed El Amine Youmbai}
	\address{LABTHOP Laboratory, Mathematics Department, Faculty of Exact Sciences, University of El Oued, PO Box 789, 39000 Echott El Oued, Algeria}
	\email{youmbai-amine@univ-eloued.dz}
	
	
	\author{Arman Shamsi Zargar}
	\address{Department of Mathematics and Applications, University of Mohaghegh Ardabili, Ardabil, Iran}
	\email{zargar@uma.ac.ir}
	
	\author{Maksym Voznyy}
	\address{Department of Technology, Stephen Leacock CI, Toronto District School Board, 
		Toronto, Canada}
	\email{maksym.voznyy@tdsb.on.ca}
	
	\subjclass[2020]{Primary 14H52, 11D25; Secondary 11D09, 05A17}
	
	
	\keywords{triples of integers, equal sum of integers, equal products, diophantine equation, partition, parametric solution, elliptic curve}

\begin{abstract}
	Let $\Sl(M,N)$ denote a set of $\ell$ triples of positive integers having the same sum $M$ and the same product $N$. For each $2\leq\ell\leq 4$ we establish a connection between a subset of $\Sl(M,N)$ with (integral) parametric elements and a family of elliptic curves. When $\ell=2$ and $3$, we use certain known subsets of $\Sl(M,N)$ with parametric elements and respectively find families of elliptic curves of generic rank~$\geq 5$ and $\geq 6$, while for $\ell=4$ we first obtain a subset of $\Sl(M,N)$ with parametric elements, then construct a family of elliptic curves of generic rank $\geq 8$. Finally, we perform a computer search within these families to find specific curves with rank~$\geq 11$  and in particular we found two curves of rank~$14$.
\end{abstract}
\maketitle

\section{Introduction}

For any positive integers $\ell$ and $n$, let $\Sln(M,N)$ denote a set of $\ell$ triples $(x_{j1}^n,x_{j2}^n,x_{j3}^n)$ of positive integers having the same sum $M$ and the same product $N$. For the $j$-th element  of the set, we define
$$\begin{aligned}
	M_j^{(n)}&=\sum_{i=1}^{3} x_{ji}^n, & N_j^{(n)}&=\prod_{i=1}^{3} x_{ji}^n, 
	& T_j^{(n)}&=\sum_{1\leq i<k\leq 3} x_{ji}^nx_{jk}^n, & (1\leq j\leq \ell).
\end{aligned}$$ 
We drop the notation ``$(n)$" from the terminology when $n=1$. 

In this context, we are interested in the following partitioning problem, which from the geometric point of view, is equivalent to examining the existence of $\ell$ rectangular boxes with integer sides having the same perimeter and the same volume:
\begin{problem}
	\label{problem}
	For given positive integers $\ell\geq 2$ and $n$, construct a subset of $\Sln(M,N)$ with (integral) parametric elements.
\end{problem}
Note that solving this problem is equivalent to finding a parametric solution to the symmetric system of diophantine equations 
\begin{equation}
	\label{Sln}
	\begin{aligned}
		M&=M_j^{(n)}, &	N&=N_j^{(n)},
	\end{aligned}
\end{equation}
for all $j=1,\ldots,\ell$.

Ever since Motzkin's conjecture, the aforesaid problem has been subjected to investigations by some authors (see \cite{C-Y}, \cite[D16]{Guy} and \cite{Mauldon}).  In 1989, Kelly \cite{KellyII} showed the existence of an arbitrarily large number of such triples. In 1996, Schinzel \cite{Schinzel} reproved the same result in a different way. None of these two proofs was effective for constructing a numerical example until in 2012, Choudhry \cite{ChoudhryTriads} gave a constructive method. It is pertinent to note that in 2013, Zhang and Cai \cite{Zhang-Cai} generalized the aforesaid result of Kelly from triples to $m$-tuples. 

To the best knowledge of the authors, Problem~\ref{problem} has been effectively solved only for the $(\ell,n)$'s given in the following table:
\begin{table}[htbp]
	\caption{Known parametric solutions to \eqref{Sln}}\label{tab: known solutions}
	\begin{tabular}{cc}
		\toprule
		$(\ell,n)$ &  References in the chronological order of discovery  \\ 
		\midrule
		$(2,1)$  &  \cite{Moessner}, \cite{Gloden}, \cite{Chou-acta}, \cite{Choudry2par},  \cite{Chou-Rocky2013}\\ [5pt]
		$(2,2)$ & \cite{Bini1909}, \cite{Gloden}, \cite{Zeit}, \cite{Chou-acta}, \cite{Kelly-AMS1991}, \cite{Chou-Ganita}, \cite{Chou-Wrob}, \cite{Choud-JIS-2021}  \\ [5pt]
		$(2,3)$ &   \cite{Gloden}, \cite{Choud-ms-2001}  \\ [5pt]
		$(2,4)$ & 	\cite{ChoudTriads2001}, \cite{Chou-biquadrates2017} \\ [5pt]
		$(3,1)$ & \cite{Sadek} \\ [5pt]
		$(3,2)$ & \cite{Chou-Wrob} \\
		\bottomrule
	\end{tabular}
\end{table}

On the other direction, making any connection between algebraic or geometric notions and elliptic curves has been of interest to researchers. Specifically, symmetric diophantine equations of certain forms have led to families of elliptic curves with higher ranks, see, for example, \cite{A-P, Chou-Zar23, Duje2002}. As far as we are aware, there are only a few works in  literature that have studied some connections between elliptic curves of positive rank and triples satisfying~\eqref{Sln}. In 1989, Kelly \cite{KellyII} illustrated such a connection between positive rank elliptic curves of the form $y^2-Mxy-Ny=x^3$ and triples satisfying~\eqref{Sln} for $(\ell,n)=(1,1)$. In 2015, Sadek and El-Sissi \cite{Sadek} extended this result by studying the same family assigned to triples satisfying~\eqref{Sln} for $(\ell,n)=(2,1)$ and $(3,1)$ and showed that the generic ranks of the associated families are $\geq 2$ and $\geq 3$, respectively. In 2019, Choudhry \cite{Choudhry} constructed several families of elliptic curves whose generic ranks range from $\geq 8$ to $\geq 12$ coming from the system $M=M_j$ and $T=T_j$ for some $j$'s. 

In this work, we establish a new connection between each of certain solutions of \eqref{Sln} for $2\leq \ell \leq 4$ with $n=1$ and a family of elliptic curves, whose torsion subgroups are trivial in general. When $(\ell,n)=(2,1)$ and $(3,1)$, we use the solutions given in \cite{Chou-Rocky2013} and \cite{Sadek} and correspondingly introduce a family of generic rank $\geq 5$ and $\geq 6$. For the case $(\ell,n)=(4,1)$, we first obtain a parametric solution of \eqref{Sln} and then construct a family of elliptic curves of generic rank~$\geq 8$. Finally, we perform a computer search within those families to find specific curves with high rank and in particular we found two curves of rank~$14$

\section{Families of elliptic curves related to \eqref{Sln}}

First we recall the following specialisation theorem of N\'{e}ron \cite[Theorem~20.1]{Silverman} that we need for our next results:
\begin{theorem}
	\label{Neron}
	Let $K$ be a number field, and let $E$ be an elliptic curve defined over the function field $K(\Pn)$. Then there are infinitely many points $t\in\Pn(K)$ such that the specialisation homomorphism 
	$$\sigma_t: E(K(\Pn)) \rightarrow E_t(K)$$
	is injective. The set of $t$ for which $\sigma_t$ is noninjective forms a thin set.
\end{theorem}

\subsection{A family with rank~$\geq 5$} \label{sectionA} According to \cite[Subsection~3.1]{Chou-Rocky2013}, Choudhry has fully determined the set ${\mathcal{S}}_2(M,N)$ by finding two complete solutions to the related equations~\eqref{Sln}. Without loss of generality, we consider his first solution, then 
$${\mathcal{A}}:={\mathcal{A}}(M,N):={\mathcal{S}}_2(M,N)=\{(a_1,a_2,a_3), (b_1,b_2,b_3)\}$$
where
\begin{equation*}
	\begin{aligned}
		a_1&=p(s+rt), &	a_2&=q(s+pt), & a_3&=r(s+qt), \\
		b_1&=q(s+rt), & b_2&=r(s+pt), & b_3&=p(s+qt).
	\end{aligned}
\end{equation*}
To this set, we assign an elliptic curve described by the equation
$$y^2=(x+a_1 a_2)(x+a_2 a_3)(x+a_1 a_3)+e^2_1 x^2,$$
where $e_1$ is a nonzero rational number which will be defined later. 

The above elliptic curve can be rewritten as:
\begin{equation}
	\label{E2}
	y^2=x^3+(T_1+e_1^2)x^2+MNx+N^2,
\end{equation}
where $T_1=a_1a_2+a_1a_3+a_2a_3$. 

By imposing 
\begin{equation}
	\label{condition1}
	T_1 + e_1^2= T_2 + e_2^2,
\end{equation}
for some rational $e_2$, the curve~\eqref{E2} will contain the seven rational points $P_1=(0, N)$, $P_{ij}=(-a_ia_j, a_ia_je_1)$ and $Q_{ij}=(-b_ib_j, b_ib_je_2)$, $1\leq i<j\leq 3$. Here, $T_2=b_1b_2+b_1b_3+b_2b_3$.

The quadratic~\eqref{condition1} is easily accomplished by parametrising
\begin{equation*}
	\label{e1}
	e_1=\frac{k^2+T_2-T_1}{2k}, \  e_2=\frac{k^2+T_1-T_2}{2k},
\end{equation*}
for any nonzero rational number $k$. 

By this description, we have thus constructed a family of elliptic curves, defined over $\Q(p,q,r,s,t,k)$, coming from the elements of ${\mathcal{A}}$, that has the above rational points. We denote this family by ${\mathcal E}_{{\mathcal{A}}}$.  

\begin{remark}
	Notice that the points $P_{ij}$ are co-linear (because their $(x,y)$-coordinates satisfy the linear equation $y+e_1x=0$), showing that at most two of them can be linearly independent. The same result holds for the points $Q_{ij}$. Therefore, at most five points of the seven points of ${\mathcal E}_{\mathcal A}$ can be linearly independent. 
\end{remark}

In the next theorem we show that five of seven points of $E_{\mathcal A}$ are linearly independent for infinitely many tuples $(p,q,r,s,t,k)$. 

\begin{theorem}\label{th1}
	For the set of all tuples $(p,q,r,s,t,k)$ except for a thin subset, the rank of the family ${\mathcal E}_{\mathcal{A}}$ is at least five with the five linearly independent points $P_1$, $P_{12}$, $P_{13}$, $Q_{12}$, $Q_{13}$.
\end{theorem}
\begin{proof}	
	By N\'{e}ron's specialisation theorem (Theorem~\ref{Neron}), in order to prove that the family ${\mathcal E}_{\mathcal{A}}$ has rank~$\geq 5$ over $\Q(p,q,r,s,t,k)$, it suffices to find a specialisation $(p,q,r,s,t,k)=(p_0,q_0,r_0,s_0,t_0,k_0)$ such that the above points in the statement are linearly independent on the specialised curve over $\Q$. We take $(p,q,r,s,t,k)=(1, 4, 2, 4, 1, 1)$ for which we have 
	$${\mathcal{A}}(42,1920)=\{(6, 20, 16), (24, 10, 8)\}$$
	and $$T_1=536, \  T_2=512, \ e_1= -\frac{23}{2}, \ e_2= \frac{25}{2}.$$
	The five points are linearly independent and of infinite order on the specialised rank~$5$ elliptic curve
	$${\mathcal E}_{\mathcal{A}(42,1920)}: y^2 = x^3 + \frac{2673}{4}x^2 + 80640x + 3686400.$$
	Indeed, the determinant of the N\'{e}ron--Tate height pairing matrix (a.k.a. regulator) of the specialised five points with $x$-coordinates
	$$x(P_1)=0, \ x(P_{12})=-120, \ x(P_{13})=-96, \ x(Q_{12})=-240, \ x(Q_{13})=-192,$$
	is the nonzero value $23.4808049005680$ as computed by SageMath \cite{Sage}. This shows that the family of elliptic curves ${\mathcal E}_{\mathcal{A}}$ has rank~$\geq 5$ over $\Q(p,q,r,s,t,k)$ with independent points $P_1$, $P_{12}$, $P_{13}$, $Q_{12}$, $Q_{13}$ (except for a thin subset of $(p,q,r,s,t,k)$'s).  
\end{proof}

\begin{remark}
	Corresponding to each of the known subsets $\mathcal S\neq \mathcal A$ of $\Sln(M,N)$ with $\ell\geq 2$, being found in the references mentioned in Table~\ref{tab: known solutions} and in Subsection~\ref{sectionC} (i.e., the set $\mathcal C$), one can assign a family of elliptic curves similar to ${\mathcal E}_{\mathcal{A}}$ and accordingly establish a result as that of  Theorem~\ref{th1}. 
\end{remark}

\subsection{A family with rank~$\geq 6$}  \label{sectionB} From \cite[Section~4]{Sadek}, we consider the subset ${\mathcal{B}}:={\mathcal{B}}(M,N)$ of ${\mathcal{S}}_3(M,N)$
with the elements  
$$\begin{aligned}
	(a_1,a_2,a_3)&=(pqrw,s,z), & (b_1,b_2,b_3)&=(w,qrs,pz), & (c_1,c_2,c_3)&=(pw,qs,rz),
\end{aligned}$$
where
$$
\begin{array}{c}
	s=pqr(r-p)+p^2-p-r+1, \quad w=q(r^2-p-r+1)+p-r, \vspace{.15cm}\\ \text{and} \quad z=pqr(qr-q-1)+p+q-1.
\end{array}
$$
Then, we set 
$$f_1(x)=\prod_{i=1}^{3}(x+a_i), \ 
f_2(x)=\prod_{i=1}^{3}(x+b_i), \ 
f_3(x)=\prod_{i=1}^{3}(x+c_i),
$$
so that 
\begin{equation}
	\label{f123}
	f_1(x)=f_i(x)+(T_1-T_i)x, \ i=2,3.
\end{equation}

We now introduce the quartic elliptic curve
\begin{equation}
	\label{ell3}
	\begin{aligned}
		E: y^2&=Axf_1(x)+B^2x^2, \\
		&=Ax(f_2(x)+(T_1-T_2)x)+B^2x^2, \\
		&=Ax(f_3(x)+(T_1-T_3)x)+B^2x^2,
	\end{aligned}
\end{equation}
where  
$$T_1=\sum_{1\leq i<j\leq 3}a_ia_j, \ T_2=\sum_{1\leq i<j\leq 3}b_ib_j, \ T_3=\sum_{1\leq i<j\leq 3}c_ic_j,$$
and the coefficients $A$ and $B$ will be given later. Note that the last two equalities in \eqref{ell3} come from \eqref{f123}. By the first equality of \eqref{ell3}, the elliptic curve $E$ clearly contains the three rational points with abscissas $-a_1$, $-a_2$, $-a_3$. Besides, by the second and third equality of \eqref{ell3}, the curve $E$ has the additional rational points with abscissas $-b_i$ and $-c_i$, $i=1,2,3$, if and only if the system of equations 
$$A(T_1-T_{i+1})=m_i^2-B^2, \ i=1,2,$$
is solvable. This system has the following parametric solution
\begin{equation}
	\label{ABmi}
	\begin{aligned}
		A&=-4hk(h-k)\left((T_1-T_2)h-(T_1-T_3)k\right),\\
		B&=(T_1-T_2)h^2-(T_1-T_3)k^2, \\
		m_i&=(T_1-T_2)h^2-2(T_1-T_{i+1})hk+(T_1-T_3)k^2, \ i=1,2,
	\end{aligned}
\end{equation}
for any nonzero, unequal rationals $h,k$. 

Henceforth, we have constructed a quartic family of elliptic curves, defined over $\Q(p,q,r,h,k)$, coming from the elements of $\mathcal{B}$, that has the nine points
$$P_i=(-a_i, a_iB), \ Q_i=(-b_i, b_im_1), \ R_i=(-c_i, c_im_2), \ i=1, 2, 3,$$
where $A$, $B$ and $m_i$'s are given in \eqref{ABmi}. 

The quartic elliptic curve $E$ is birationally equivalent to the cubic elliptic curve ${\mathcal E}_{\mathcal B}$ given by the equation
\begin{equation*}
	\label{cubicell3}
	Y^2=X^3+(AT_1+B^2)X^2+A^2M_iN_iX+A^3N_i^2, \ i=1,2,3,
\end{equation*}
via the rational transformation $X=AN_i/x$, $Y=AN_iy/x^2$. 

In the view of this transformation, the nine points $P_i$, $Q_i$ and $R_i$, $i=1,2,3$, are mapped respectively to the points ${\mathcal P}_i$, ${\mathcal Q}_i$ and ${\mathcal R}_i$, given below,
$$
\begin{aligned}
	{\mathcal P}_1&=(-a_2a_3A, a_2a_3AB), &  {\mathcal Q}_1&=(-b_2b_3A, b_2b_3Am_1), & 	{\mathcal R}_1&=(-c_2c_3A, c_2c_3Am_2), \\
	{\mathcal P}_2&=(-a_1a_3A, a_1a_3AB), &  {\mathcal Q}_2&=(-b_1b_3A, b_1b_3Am_1), & 	{\mathcal R}_2&=(-c_1c_3A, c_1c_3Am_2), \\
	{\mathcal P}_3&=(-a_1a_2A, a_1a_2AB), &  {\mathcal Q}_3&=(-b_1b_2A, b_1b_2Am_1), & 	{\mathcal R}_3&=(-c_1c_2A, c_1c_2Am_2).
\end{aligned}
$$

\begin{remark}
	Notice that the points ${\mathcal P}_{i}$ are co-linear (because their $(x,y)$-coordinates satisfy the linear equation $x+By=0$), showing that at most two of them can be linearly independent. The same result holds for each of the points ${\mathcal Q}_{i}$ and ${\mathcal R}_{i}$. Therefore, at most six of the nine points of ${\mathcal E}_{\mathcal{B}}$ can be linearly independent. 
\end{remark}

In the next theorem we show that six of the nine points of ${\mathcal E}_{\mathcal B}$ are linearly independent for infinitely many tuples $(p,q,r,h,k)$. 

\begin{theorem}\label{th2}
	For the set of all tuples $(p,q,r,h,k)$ except for a thin subset, the rank of the family 
	${\mathcal E}_{\mathcal{B}}$ is at least six with the six linearly independent points ${\mathcal P}_i$, $\mathcal{Q}_i$, ${\mathcal R}_i$, $i=1,2$.
\end{theorem}
\begin{proof}	
	Take the specialisation $(p,q,r,h,k)=(2, 2, 4, 1, -1)$ which makes
	$s=31$, $z=83$, $w=20$, and hence we have 
	$${\mathcal{B}}(434,823360)=\{(320, 31, 83), (20, 248, 166), (40, 62, 332)\}$$
	and
	$$
	\begin{array}{c}
		T_1= 39053, \quad T_2= 49448, \quad T_3= 36344, \vspace{.2cm}\\ 
		A= -61488, \quad B= -13104, \quad m_1=-28476, \quad m_2= -2268.
	\end{array}
	$$
	The regulator of the six specialised points with $X$-coordinates
	$$
	\begin{aligned}
		X({\mathcal P}_1)& = 158208624, & X({\mathcal Q}_1)&= 2531337984, & X({\mathcal R}_1)&= 1265668992,\\
		X({\mathcal P}_2)& = 1633121280, & X({\mathcal Q}_2)&= 204140160, & X({\mathcal R}_2)&= 816560640,
	\end{aligned}
	$$
	on the specialised elliptic curve 
	$$\begin{aligned}
		{\mathcal{E}}_{{\mathcal{B}(434,823360)}}:	Y^2 &= X^3 - 2229576048X^2 + 1351015178454466560X \\
		&\qquad - 157597974109784775013171200
	\end{aligned}$$
	is the nonzero value $534.520794417629$, carried out by SageMath. The result follows from Theorem~\ref{Neron}.
\end{proof}

\begin{remark}
	Corresponding to each of the known subsets $\mathcal S\neq \mathcal B$ of $\Sln(M,N)$ with $\ell\geq 3$, being found in \cite{Sadek} and \cite{Chou-Wrob} (cf.~Table~\ref{tab: known solutions}) and in Subsection~\ref{sectionC}, one can assign a family of elliptic curves similar to ${\mathcal E}_{\mathcal{B}}$ and accordingly  establish a result as that of Theorem~\ref{th2}. 
\end{remark}

\subsection{A family with rank~$\geq 8$}  \label{sectionC}
For $(\ell,n)=(4,1)$, there is no known parametric solution of \eqref{Sln}. We first begin with presenting such a parametric solution. 
\begin{proposition} \label{ours}
	For arbitrary integers $q$ and $s$ with $q\neq 0, s^2$, $s\neq 0, 1$, let 
	$$\begin{aligned}
		t_1 &= q^2s^2-2q^2s+qs+q-1, \\
		t_2 &= q^2s^2-2q^2s+qs^2-qs+s^2+q-s, \\
		t_3 &= q^3s^3-q^2s^4-4q^3s^2+5q^2s^3-qs^4+4q^3s-6q^2s^2+2qs^3-s^4+q^2s+2s^3\\
		& \quad \quad -2q^2+qs-2s^2+q,\\
		t_4&= -q^2s^5+q^3s^3+4q^2s^4-4q^3s^2-4q^2s^3-2qs^4+4q^3s+q^2s^2+4qs^3-q^2s\\
		&\quad \quad -s^3-2q^2-qs+s^2+2q-s, \\ 
		t_5&= q^2s^7+q^4s^4-3q^3s^5-4q^2s^6+qs^7-4q^4s^3+15q^3s^4-qs^6+4q^4s^2-22q^3s^3\\
		&\quad \quad +12q^2s^4-5qs^5+2s^6+10q^3s^2-12q^2s^3+7qs^4-4s^5-4q^3s\\
		&\quad \quad +12q^2s^2-7qs^3+4s^4-4q^2s+qs^2-s^3+q^2.
	\end{aligned}
	$$
	Then,
	$$
	\begin{aligned}
		(a_1,a_2,a_3)&=\big((s-1)t_1t_2t_3, sq(1-s)t_1t_5, (s^2-q)t_1^2t_4\big), \\
		(b_1, b_2, b_3)&=\big((s-1)qt_1t_2t_3, -t_1^2t_5, (s-1)s(s^2-q)t_1t_4\big), \\
		(c_1,c_2,c_3)&=\big((1-s)t_1t_5, (s-1)q t_1t_2t_4, s(s^2-q)t_1^2t_3\big), \\
		(d_1, d_2, d_3)&=\big((1-s)qt_1t_5, (s-1)st_1t_2t_4, (s^2-q)t_1^2t_3 \big),
	\end{aligned}
	$$
	is a two-parameter solution to \eqref{Sln} with $(\ell,n)=(4,1)$.
\end{proposition}
\begin{proof}
	Consider $(a_1,a_2,a_3)=(p,wqs,rz)$, $(b_1,b_2,b_3)=(pq,rw,sz)$, $(c_1,c_2,c_3)=(w,qr,spz)$ and $(d_1,d_2,d_3)=(wq,rs,pz)$. Then, clearly the products of the components of these triples are equal. Solving the equations $M_1=M_2$, $M_2=M_3$, $M_3=M_4$ gives three formulas for $w$. By the equality between the first two and the second two of the resulting $w$'s, we get two formulas for $z$. The equality between these $z$'s gives rise to
	\begin{multline*}
		(q^2s^2-2q^2s+qs-rs+q+r-1)p^2\\
		+(-qrs^3+q^2rs+r^2s^2+q^2r-qr^2-qr-qs+s)p\\
		-r(q^2s^2-qs^3+q^2r-q^2s-qr^2+r^2s-rs+s^2)=0.
	\end{multline*} 
	By putting $q^2s^2-2q^2s+qs-rs+q+r-1=0$ in the latter equation, it follows that
	\begin{equation}
		\label{r}
		r = \frac{q^2s^2-2q^2s+qs+q-1}{s-1}
	\end{equation}
	and
	\begin{equation}\begin{aligned}
			\label{p}
			p =-r\frac{q^2s^2-qs^3+q^2r-q^2s-qr^2+r^2s-rs+s^2}{qrs^3-q^2rs
				-r^2s^2-q^2r+qr^2+qr+qs-s}.
	\end{aligned}\end{equation}
	Now, by substituting \eqref{r} in \eqref{p} and one of the obtained $z$'s and $w$'s, we get a rational solution of \eqref{Sln} for $(\ell,n)=(4,1)$ in terms of $q$ and $s$. By scaling we get the desired result as given in the statement of the theorem.
\end{proof}

Now, we consider the subset $\mathcal C:=\mathcal{C}(M,N)$ of ${\mathcal{S}}_{4}(M,N)$ with the elements introduced in Proposition~\ref{ours}. Then, we set 
$$f_1(x)=\prod_{i=1}^{3}\{(x+a_i)(x+b_i)\}, \  
f_2(x)=\prod_{i=1}^{3}\{(x+c_i)(x+d_i)\},
$$
so that 
\begin{equation}
	\label{f1234}
	f_1(x)-f_2(x)= Tx^4+MTx^3+(T_1T_2-T_3T_4)x^2+NTx,
\end{equation}
where 
$$T_1=\sum_{1\leq i<j\leq 3}a_ia_j, \ T_2=\sum_{1\leq i<j\leq 3}b_ib_j, \ T_3=\sum_{1\leq i<j\leq 3}c_ic_j, \ T_4=\sum_{1\leq i<j\leq 3}d_id_j,$$
and $T=T_1+T_2-T_3-T_4$. 

One can rewrite \eqref{f1234} as $\phi^{2}_{1}(x)-\phi^{2}_{2}(x)$ where
$$\begin{aligned}
	\phi_{1}(x)&=x^3+Mx^2+\frac{4(T_{1}T_{2}-T_{3}T_{4})+T^2}{4T}x+N, \\
	\phi_{2}(x)&= x^3+Mx^2+\frac{4(T_{1}T_{2}-T_{3}T_{4})-T^2}{4T}x+N.
\end{aligned}
$$
It follows that
$$\phi^{2}_{i}(x)-f_i(x)=K x^4+MK x^3+H x^2+NKx, \ i=1,2,$$
where 
$$\begin{aligned}
	K &= -\frac{T^2+2(T_3+T_4)T-4(T_1T_2-T_3T_4)}{2T},\\
	H &= \frac{\left(T^2+4(T_1T_2-T_3T_4)\right)^2}{16T^2}-T_1T_2.\\
\end{aligned}$$

We now introduce the quartic elliptic curve
\begin{equation}
	\label{ell4}
	y^2=Kx^4+MKx^3+Hx^2+NKx,
\end{equation}
which contains the rational points $P_{a_i}$, $P_{b_i}$, $P_{c_i}$, $P_{d_i}$ with abscissas $-a_i$, $-b_i$, $-c_i$ and $-d_i$ for $i=1,2,3$, respectively.

The quartic curve~\eqref{ell4} reduces to the following cubic elliptic curve 
\begin{equation}
	\label{cubicell4}
	{\mathcal E}_{\mathcal C} : Y^2=X^3+HX^2+MNK^2X+N^2K^3,
\end{equation}
by the rational transformation $X=NK/x$, $Y=NKy/x^2$. In the view of this transformation, the above twelve points are sent respectively to the points ${\mathcal P}_{a_i}$, ${\mathcal P}_{b_i}$, ${\mathcal P}_{c_i}$, ${\mathcal P}_{d_i}$ with $X$-coordinates 
$$-\frac{NK}{a_{i}}, \ -\frac{NK}{b_{i}}, \ -\frac{NK}{c_{i}}, \ \text{and} \  -\frac{NK}{d_{i}}  \quad \text{for} \ i=1,2,3.$$

We have thus constructed a family of elliptic curves, defined over $\Q(q,s)$, coming from the elements of $\mathcal{C}$, that has the above twelve points.

\begin{remark}
	Notice that the points ${\mathcal P}_{a_i}$ are co-linear, showing that at most two of them can be linearly independent. The same result holds for each of the points ${\mathcal P}_{b_i}$, ${\mathcal P}_{c_i}$ and ${\mathcal P}_{d_i}$. Therefore, at most eight points of the twelve points of ${\mathcal{E}}_{\mathcal{C}}$ can be linearly independent. 
\end{remark}

In the next theorem we show that eight of the twelve points of ${\mathcal E}_{\mathcal C}$ are linearly independent for infinitely many tuples $(q,s)$. 

\begin{theorem}\label{th3}
	For the set of all tuples $(q,s)$ except for a thin subset, the rank of the family ${\mathcal E}_{\mathcal{C}}$ is at least eight with the eight linearly independent points ${\mathcal P}_{a_i}$, $\mathcal{P}_{b_i}$, ${\mathcal P}_{c_i}$, ${\mathcal P}_{d_i}$, $i=1,2$.
\end{theorem}
\begin{proof}	
	We specialise at $(q,s)=(3,2)$ which makes
	$$
	t_1=8, \quad t_2= 11, \quad t_3= 1, \quad t_4= -6, \quad t_5= 29,
	$$ 
	and hence we have 
	\begin{multline*}
		{\mathcal{C}}(-1688,47038464)=\{(88, -1392, -384),  (264, -1856, -96), \\
		(-232, -1584, 128),  (-696, -1056, 64)\}
	\end{multline*}
	and
	$$
	\begin{array}{c}
		T_1= 378240, \quad T_2= -337152, \quad T_3= 135040, \quad T_4= 622848, \quad 
		T= -716800,  \vspace{.2cm}  \\ 
		K=191008, \quad H= 140991510784.
	\end{array}
	$$
	The regulator of the eight points with $X$-coordinates
	$$
	\begin{aligned}
		X({\mathcal{P}}_{a_1})&=-102099124224, & X({\mathcal{P}}_{c_1})&=38727254016, \\
		X({\mathcal{P}}_{a_2})&=6454542336, & X({\mathcal{P}}_{c_2})&= 5672173568, \\
		X({\mathcal{P}}_{b_1})&=-34033041408, & X({\mathcal{P}}_{d_1})&= 12909084672, \\
		X({\mathcal{P}}_{b_2})&=4840906752, & X({\mathcal{P}}_{d_1})&= 8508260352, 
	\end{aligned}
	$$
	on the specialised elliptic curve 
	$$
	\begin{aligned}
		{\mathcal{E}}_{\mathcal{C}}:	Y^2 &= X^3 + 140991510784 X^2 - 2896867880665872334848 X\\
		&\qquad + 15419167818458889008922652311552
	\end{aligned}
	$$
	is the nonzero value $15150.2483213544$, carried out by SageMath. The result follows now from Theorem~\ref{Neron}.
\end{proof}

\section{Heuristics}

As all curves in the families ${\mathcal{E}}_{\mathcal A}$, ${\mathcal{E}}_{\mathcal B}$ and ${\mathcal{E}}_{\mathcal C}$ have trivial torsion subgroups in general, to identify promising high-rank candidates we used heuristic arguments (motivated by the BSD conjecture) given by Mestre~\cite{Mestre} and Nagao~\cite{Nagao}. They suggest that for curves of high rank, certain sums should assume the largest values in the observed families. The first of these sums is $$S_1(X)=\sum_{p \le X}\frac{N_p+1-p}{N_p}\log p $$ where $X$ is a prime bound for primes $p$, and $N_p=|E(\mathbb{F}_p)|$ is the number of points on elliptic curve $E$ under the reduction modulo $p$. We followed the sieve phase of the general method for finding high-rank elliptic curves described in \cite[pp.~64--68]{Duje}.

The Mestre--Nagao sum $S_1(X)$ with the prime bound $X = 10^{6}$ was calculated in Magma \cite{Magma} for all curves with the absolute values of parameters $p$,$q$,$r$,$s$,$t$,$k$,$h$ up to $40$. For curves with $S_1(10^{6}) \geq 100$, {\sf{TwoSelmerGroup}} function in Magma was used to deduce the upper bound $R$ on the rank of the Mordell--Weil group of the curve. For all curves with $R \geq 11$, {\sf{DescentInformation}} function in Magma was used to uncover the independent generators of the Mordell--Weil group of the curve.

Parameters for some high-rank curves are listed in the next tables. Note that all listed curves are non-isomorphic.

We were also able to find a number of curves of rank $15$ in the form \eqref{cubicell4} where the set of four triples $\{(a_1,a_2,a_3), (b_1,b_2,b_3), (c_1,c_2,c_3), (d_1,d_2,d_3)\}$ was determined by a direct search given a fixed sum $M$ rather than using Proposition \ref{ours}. The first curve of rank $15$ occurs for $M = 4907$ $(N = 1628394768)$ whose respective set is  
$$\{(632, 726, 3549), (312, 2054, 2541), (507, 924, 3476), (474, 1001, 3432)\};$$ the minimal model of the curve is
$$
\begin{aligned}
	y^{2} + xy = x^{3} & - 901882569760647935195484561648738x \\
	& + 13932920298020241870290850727467604223423326273092.
\end{aligned}
$$

Magma's special function {\sf{DescentInformation}} was able to find $15$ independent points on this curve in $184$ core-hours on Intel Core i$7$-$8700$ CPU, although our knowledge of the $8$ independent points on \eqref{cubicell4} reduced the search time to $1.6$~core-hours (over a hundredfold calculation speedup).

In that light, Proposition~\ref{ours} and a rank $13$ curve obtained from rational $(q, s)$ in the last table present significant value.

\begin{table}[htbp]
	\caption{High rank curves ${\mathcal{E}}_{\mathcal A}$}\label{table1}
	\begin{tabular}{lc lc}
		\toprule
		$(p,q,r,s,t,k)$ & Rank & $(p,q,r,s,t,k)$ & Rank  \\ 
		\midrule
		$(-71, -54, -36, 14, 36,  1)$ & $14$ & $( 32,  41, 55, 60, 46,  1)$ & $14$  \\
		\midrule
		$(-37, -36,  37, 25, 36, 37)$ & $13$ & $(-33, -29, 17,  6, 13,  9)$ & $13$ \\ 	
		$(-36, -33,  31,  6, 23,  7)$ & $13$ & $(-30, -27, 12, 16, 13, 17)$ & $13$ \\
		$(-35, -34,  17, 29, 14, 13)$ & $13$ & $(-25, -15, 20,  8, 21, 22)$ & $13$ \\ 	
		$(-35, -31,   8, 12, 18, 11)$ & $13$ & $(-22, -11, 22, 13, 22,  2)$ & $13$ \\ 	
		\midrule
		$(-22, -18,   4, 20,  9, 19)$ & $12$ & $(-19, -18, -5,  6, 19,  1)$ & $12$ \\
		$(-21, -13,  11, 14, 15, 19)$ & $12$ & $(-18,   6, 12, 16,  6, 11)$ & $12$ \\
		$(-20, -16,  13,  3, 16,  7)$ & $12$ & $(-17, -10,  9,  1, 14,  9)$ & $12$ \\
		$(-19,   8,  16,  7,  9, 16)$ & $12$ & $(-16,  -5, 20,  2, 25,  1)$ & $12$ \\
		\bottomrule
	\end{tabular}
\end{table}

\begin{table}[htbp]
	\caption{High rank curves ${\mathcal{E}}_{\mathcal B}$}\label{table3}
	\begin{tabular}{lclc}
		\toprule
		$(p,q,r,k,h)$ & Rank & $(p,q,r,k,h)$ & Rank \\ 
		\midrule
		$(-12,   7,  -3, 15, -1)$ & $13$ & $(-1, -8,  -6,  5,   4)$ & $13$ \\
		$( -4,   5, -13, -1,  1)$ & $13$ & $( 4, -3, -14, 13, -14)$ & $13$ \\
		$( -2, -13,  11,  1,  6)$ & $13$ & $(11, -2, -12,  4,  -3)$ & $13$ \\
		$( -2,   8,  15, 11, -4)$ & $13$ & $(14, -2, -13,  1,   8)$ & $13$ \\
		\midrule
		$(-11,   5,   3,  5,  6)$ & $12$ & $(-3, 10,   3,  2,  -5)$ & $12$ \\
		$( -8,  -6,   9,  5, -7)$ & $12$ & $(-2, 10,   2,  8,  -7)$ & $12$ \\
		$( -7,   3,   2,  8, -5)$ & $12$ & $(-1,  7,  -9,  1,  -6)$ & $12$ \\
		$( -5,  -6,   5, -3,  4)$ & $12$ & $( 3, -7,   9,  7,   6)$ & $12$ \\
		\midrule
		$( -6,   3,  -1, -2,  6)$ & $11$ & $(-2,  4,   3, -4,   5)$ & $11$  \\
		$( -5,  -5,   5, -2, -3)$ & $11$ & $( 2,  3,   6, -5,  -2)$ & $11$  \\ 
		$( -4,  -3,   4, -2,  6)$ & $11$ & $( 2,  4,  -5, -3,  -4)$ & $11$  \\
		$( -3,  -4,   3, -3, -5)$ & $11$ & $( 5, -2,  -5,  5,  -1)$ & $11$  \\
		\bottomrule
	\end{tabular}
\end{table}

\begin{table}[htbp]
	\caption{High rank curves ${\mathcal{E}}_{\mathcal C}$}\label{table5}
	\begin{tabular}{lc lc}
		\toprule
		$(q,s)$ & Rank & $(q,s)$ & Rank \\ 
		\midrule
		$(7/11, -1)$ & $13$ && \\
		\midrule
		$(-12/11,  2)$ & $12$ & $(-1/4, -1)$ & $12$ \\
		$(-9/5, -2/5)$ & $12$ & $(3/4, 2)$ & $12$ \\
		$(-8, -7/2)$ & $12$ & $(9/5, -1)$ & $12$ \\
		$(-2/15, 2)$ & $12$ & $(12/7, 14/9)$ & $12$ \\
		\midrule
		$(-11/7, -1/3)$ & $11$ & $(-4/13, 2)$ & $11$ \\
		$(-9, -1)$ & $11$ & $(-2/3, 3)$ & $11$ \\
		$(-5/7, -1/2)$ & $11$ & $(1/7, -1)$ & $11$ \\
		$(-4/9, 2)$ & $11$ & $(11/8, 6/5)$ & $11$ \\
		\bottomrule
	\end{tabular}
\end{table}

\section{A concluding remark}
In this work, we considered the system of diophantine equations 
$$\sum_{i=1}^{3} x_{1i}= \cdots = \sum_{i=1}^{3} x_{\ell i}, \quad \prod_{i=1}^{3} x_{1i}= \cdots = \prod_{i=1}^{3} x_{\ell i}$$
for $\ell$ ranging from $2$ to $4$, which from the geometric
point of view is equivalent to $\ell$ rectangular boxes with integer sides having the same perimeter and the same volume, and for each system we introduced a family of elliptic curves and then studied them closely. Here arises this natural question that whether there exists more than four rectangular boxes with integer sides that have the same perimeter and the same volume.

\bibliographystyle{amsplain}

\begin{thebibliography}{10}
	
\bibitem{A-P}{J.~Aguirre and J.~C.~Peral,} \textit{Sums of biquadrates and elliptic curves,} Glas.~Mat.~Ser.~III~\textbf{48(1)} (2013), 49--58.

\bibitem{Bini1909}{U.~Bini,} Interm\'{e}d.~Math.~\textbf{16} (1909), 41--43.

\bibitem{Magma}{W.~Bosma, J.~Cannon and C.~Playoust,} \textit{The Magma algebra system.~I. The user language,} J.~Symbolic~Comput.~\textbf{24(3-4)} (1997), 235--265.

\bibitem{C-Y}{B.~Cha, A.~Claman, J.~Harrington, Z.~Liu, B.~Maldonado, A.~Miller, A.~Palma, T.~W.~H.~Wong and H.~Yi,} \textit{An investigation on partitions with equal products,} 
Int.~J.~Number~Theory~\textbf{15(08)} (2019), 1731--1744. 

\bibitem{Chou-acta}{A.~Choudhry,} \textit{Symmetric diophantine systems,} Acta~Arith.~\textbf{59(3)} (1991), 291--307.

\bibitem{Chou-Ganita}{A.~Choudhry,} \textit{On triads of squares with equal
	sums and equal products,} Ganita~\textbf{49} (1998), 101--106. 

\bibitem{Choud-ms-2001}{A.~Choudhry,} \textit{Triads of cubes with equal sums and equal products,} Math.~Student~\textbf{70(1-4)} (2001), 137--143.

\bibitem{ChoudTriads2001}{A.~Choudhry,} \textit{Triads of biquadrates with equal sums and equal products,} Math.~Student~\textbf{70(1-4)} (2001), 149--152.

\bibitem{Choudry2par}{A.~Choudhry,} \textit{Some diophantine problems concerning equal sums of integers and their cubes,} Hardy--Ramanujan~J.~\textbf{33} (2010), 59--70.

\bibitem{ChoudhryTriads}{A.~Choudhry,} \textit{Triads of integers with equal sums and equal products,} Math.~Student~\textbf{81(1-4)} (2012), 185--188.

\bibitem{Chou-Rocky2013}{A.~Choudhry,} \textit{Equal sums of like powers and equal products of integers,} Rocky~Mountain~J.~Math.~\textbf{43(3)} (2013), 763--792.

\bibitem{Chou}{A.~Choudhry,} \textit{Equal sums of like powers with minimum number of terms,} Integers~\textbf{16} (2016), \#A77.

\bibitem{Chou-biquadrates2017}{A.~Choudhry,} \textit{A diophantine problem on biquadrates revisited,} Math.~Student~\textbf{87(1-2)} (2018), 129--132.

\bibitem{Choudhry}{A.~Choudhry,} \textit{Symmetric diophantine systems and families of elliptic curves of high rank,} Rocky~Mountain~J.~Math.~\textbf{49(5)} (2019), 1419--1447.

\bibitem{Choud-JIS-2021}{A.~Choudhry,} \textit{New solutions of the Tarry--Escott problem of degrees $2$, $3$ and $5$,} J.~Integer~Seq.~\textbf{24} (2021), Article 21.8.1.

\bibitem{Chou-2022}{A.~Choudhry,} \textit{Ideal solutions of the Tarry--Escott problem of degree seven,} Integers~\textbf{22} (2022), \#A115.

\bibitem{Chou-Zar23}{A.~Choudhry and A.~Shamsi~Zargar,} \textit{An octic diophantine equation and related families of elliptic curves,} Int.~J.~Number~Theory~\textbf{19(08)} (2023), 1967--1976.

\bibitem{Chou-Wrob}{A.~Choudhry and J.~Wr\'{o}blewski,} \textit{Triads of integers with equal
	sums of squares and equal products and a related multigrade chain,} Acta~Arith.~\textbf{178(1)} (2017), 87-100. 

\bibitem{Dickson}{L.~E.~Dickson,} History of the Theory of Numbers II, Chelsea Publishing Company, New York, 1920.

\bibitem{Duje2002}{A.~Dujella,} \textit{An example of elliptic curve over $\mathbf{Q}$ with rank equal to $15$,} Proc.~Japan~Acad.~Ser.~A~Math.~Sci.~\textbf{78(7)} (2002), 109--111. 

\bibitem{Duje}{A.~Dujella,} Diophantine $m$-tuples and Elliptic Curves, Springer, 2024.

\bibitem{Gloden}{A.~Gloden,} Mehrgradige Gleichungen, Noordhoff, Groningen, 1944.

\bibitem{Guy}{R.~K.~Guy,} Unsolved Problems in Number Theory, Springer-Verlag, 3rd edition, 1994. 

\bibitem{KellyII}{J.~B.~Kelly,} \textit{Partitions with equal products (II),} Proc.~Amer.~Math.~Soc.~\textbf{107(4)} (1989), 887--893.

\bibitem{Kelly-AMS1991}{J.~B.~Kelly,} \textit{Two equal sums of three squares with equal products,} Amer.~Math.~Monthly~\textbf{98(6)} (1991), 527--529.

\bibitem{Mauldon}{J.~G.~Mauldon,} \textit{Problem E~2872,} Amer.~Math.~Monthly~\textbf{88(2)} (1981) 148.

\bibitem{Mestre}{J.~F.~Mestre,} \textit{Construction d'une courbe elliptique de rang~$\geq 12$,} C.~R.~Acad.~Sci.~Paris S\'{e}r.~I.~\textbf{295} (1982), 643--644.

\bibitem{Moessner}{A.~Moessner,} \textit{Verschiedene Diophantische Probleme und numerische Identit\"{a}ten,} Tohoku~Math.~J.~\textbf{47} (1940), 188--200

\bibitem{Mordell}{L.~J.~Mordell,} Diophantine Equations, Academic press, London, 1969.

\bibitem{Nagao}{K.~Nagao,} \textit{An example of elliptic curve over $Q$ with rank $\geq 20$,} Proc.~Japan~Acad.~Ser.~A~Math.~Sci.~\textbf{69(8)} (1993), 291--293.


\bibitem{Sadek}{M.~Sadek and N.~El-Sissi,} \textit{Partitions with equal products and elliptic curves,} Osaka~J.~Math.~\textbf{52(2)} (2015), 515--527.

\bibitem{Sage}{Sage Developers,} \textit{SageMath, Sage mathematics software system (ver. 9.3),} 2021, available at \url{https://www.sagemath.org}

\bibitem{Schinzel}{A.~Schinzel,} \textit{Triples of positive integers with the same sum and the same product,} Serdica~Math.~J.~\textbf{22(4)} (1996), 587--588.


\bibitem{Silverman}{J.~H.~Silverman,} The Arithmetic of Elliptic Curves, 2nd edition, Springer, 2009.

\bibitem{Zeit}{D.~Zeitlin,} \textit{Abstracts presented to American Mathematical Society,}~\textbf{57(239)} (1988), \#843-11-50.

\bibitem{Zhang-Cai}{Y.~Zhang and T.~Cai,} \textit{$n$-Tuples of positive integers with the same sum and the same product,} Math.~Comp.~\textbf{82(281)} (2013), 617--623. 


\end{thebibliography}

\end{document}